\documentclass[letterpaper,11pt,reqno]{amsart}

\makeatletter
\usepackage{amssymb}
\usepackage{latexsym}
\usepackage{amsbsy}
\usepackage{amsfonts}
\usepackage{tipa}
\usepackage{hyperref}
\usepackage{graphicx}

\def\marginpar#1{\ignorespaces}

\DeclareMathOperator\hit{HIT}
\DeclareMathOperator\adj{adj}
\DeclareMathOperator\tr{Tr}
\DeclareMathOperator\lp{loop}
\newtheorem{theorem}{Theorem}[section]
\newtheorem{lemma}[theorem]{Lemma}
\newtheorem{proposition}[theorem]{Proposition}
\newtheorem{corollary}[theorem]{Corollary}

\newtheorem{example}[theorem]{Example}

\numberwithin{equation}{section}

\newdimen\AAdi%
\newbox\AAbo%
\def\AAk#1#2{\setbox\AAbo=\hbox{#2}\AAdi=\wd\AAbo\kern#1\AAdi{}}%

\def\eqref#1{(\ref{#1})}
\def\eqlabel#1{\def\@currentlabel{#1}}

\def\formula#1{\def\@tempa{#1}\let\@tempb\theequation\def\theequation{%
\hbox{#1}}\def\@currentlabel{(\theequation)}$$}
\def\endformula{\leqno\hbox{(\@tempa)}$$\@ignoretrue\let\theequation\@tempb}

\def\given{\hskip5\p@\relax\vrule\@width.4\p@\hskip5\p@\relax}

\newcommand{\open}[1]{%
\par\normalfont\topsep6\p@\@plus6\p@\trivlist\item[\hskip\labelsep\itshape#1%
\@addpunct{.}]\ignorespaces}

\DeclareRobustCommand{\close}[1]{%
  \ifmmode 
  \else \leavevmode\unskip\penalty9999 \hbox{}\nobreak\hfill
  \fi
  \quad\hbox{$#1$}}

\newlength{\toskip}

\makeatother

\def\roots{ROOTS}

\begin{document}

\title[Tree formulas and Kemeny's constant]{Tree formulas, mean first passage times and Kemeny's constant of a Markov chain}

\author[Jim Pitman]{{Jim} Pitman}
\address{Statistics department, University of California, Berkeley.
} \email{pitman@stat.berkeley.edu}

\author[Wenpin Tang]{{Wenpin} Tang}
\address{Statistics department, University of California, Berkeley.
} \email{wenpintang@stat.berkeley.edu}

\date{\today} 

\begin{abstract}
This paper offers some probabilistic and combinatorial insights into tree formulas for the Green function and hitting probabilities of Markov chains on a finite state space. These tree formulas are closely related to {\em loop-erased random walks} by {\em Wilson's algorithm} for random spanning trees, and to {\em mixing times} by the {\em Markov chain tree theorem}. Let $m_{ij}$ be the mean first passage time from $i$ to $j$ for an irreducible chain with finite state space $S$ and transition matrix $(p_{ij}; i, j \in S)$. It is well known that $m_{jj} = 1/\pi_j = \Sigma^{(1)}/\Sigma_j$, where $\pi$ is the stationary distribution for the chain, $\Sigma_j$ is the tree sum, over $n^{n-2}$ trees $\textbf{t}$ spanning $S$ with root $j$ and edges $i \rightarrow k$ directed towards $j$, of the tree product $\prod_{i \rightarrow k \in \textbf{t} }p_{ik}$, and $\Sigma^{(1)}:= \sum_{j \in S} \Sigma_j$. Chebotarev and Agaev \cite{Cheaga} derived further results from {\em Kirchhoff's matrix tree theorem}. We deduce that for $i \ne j$, $m_{ij} = \Sigma_{ij}/\Sigma_j$, where $\Sigma_{ij}$ is the sum over the same set of $n^{n-2}$ spanning trees of the same tree product as for $\Sigma_j$, except that in each product the factor $p_{kj}$ is omitted where $k = k(i,j,\textbf{t})$ is the last state before $j$ in the path from $i$ to $j$ in $\textbf{t}$. It follows that Kemeny's constant $\sum_{j \in S} m_{ij}/m_{jj}$ equals $ \Sigma^{(2)}/\Sigma^{(1)}$, where $\Sigma^{(r)}$ is the sum, over all forests $\textbf{f}$ labeled by $S$ with $r$ directed trees, of the product of $p_{ij}$ over edges $i \rightarrow j$ of $\textbf{f}$. We show that these results can be derived without appeal to the matrix tree theorem. A list of relevant literature is also reviewed.
\end{abstract}
\maketitle

\setcounter{tocdepth}{1}
\tableofcontents
\section{Introduction and background}
\quad In this survey paper, we review various tree formulas of a finite Markov chain, and make connections with random spanning trees and mean first passage times in the Markov chain. Most results are known from previous work, but a few formulas and statements, e.g. the combinatorial interpretation of {\em Kemeny's constant} in Corollary \ref{new}, and the formula \eqref{harmonic3},  appear here for the first time. We offer a probabilistic and combinatorial approach to these results, encompassing the closely related results of Leighton and Rivest \cite{LR,LRbis} as well as {\em Wilson's algorithm} \cite{Wilson,ProppWilson} for generation of random spanning trees.

\quad Throughout this paper, we assume that $S$ is a finite state space. Let $m_{ij}$ be the {\em mean first passage time} from $i$ to $j$ for an irreducible Markov chain $(X_n)_{n \in \mathbb{N}}$ with
state space $S$ and transition matrix $\textbf{P}:=(p_{ij}; i, j \in S)$. That is,
\begin{equation*}
m_{ij}:=\mathbb{E}_iT_j^+ \quad \mbox{for}~i,j \in S,
\end{equation*}
where $T_j^+:=\inf\{n \geq 1; X_n=j\}$ is the hitting time of the state $j \in S$, and $\mathbb{E}_i$ is the expectation relative to the Markov chain $(X_n)_{n \in \mathbb{N}}$ starting at $i \in S$. 

\quad It is well known that the irreducible chain $(X_n)_{n \in \mathbb{N}}$ has a unique stationary distribution $(\pi_j)_{j \in S}$ which is given by
$$\pi_j=1/ m_{jj} \quad \mbox{for all}~ j \in S.$$
See e.g. Levin, Peres and Wilmer \cite[Chapter $1$]{LPW} or Durrett \cite[Chapter $6$]{Durrett} for background on the theory of Markov chains. 

\quad For a directed graph $\textbf{g}$ with vertex set $S$, write $i \rightarrow j \in \textbf{g}$ to indicate that $(i,j)$ is a directed edge  of $\textbf{g}$ and call
\begin{equation*}
\Pi^{\bf P}(\textbf{g}):= \prod_{i\rightarrow j \in \textbf{g}}  p_{ij}
\end{equation*}
the $\textbf{P}$-weight of $\textbf{g}$. Each forest $\textbf{f}$ with vertex set $S$ and edges directed towards root vertices consists of some number $r$ of trees $\textbf{t}_i$ whose vertex sets partition $S$ into $r$ non-empty disjoint subsets.  Observe that if a forest $\textbf{f}$ consists of trees $\textbf{t}_1, \dots, \textbf{t}_r$, then $\textbf{f}$ has $\textbf{P}$-weight
\begin{align*}
\Pi^{\bf P}(\textbf{t}_1, \dots, \textbf{t}_r):= \Pi^{\bf P}(\textbf{f}) = \prod_{i=1}^r \Pi^{\bf P}(\textbf{t}_i).
\end{align*}
Write $\textbf{t} \rightarrow j$ to indicate that the edges of a tree $\textbf{t}$ are all directed towards a root element $j \in \textbf{t}$.
The formula
\begin{align}
\label{kirchhoff}
m_{jj} = \frac{1}{\pi_j}=\Sigma ^{(1)} / \Sigma_j,
\end{align}
where 
\begin{align}
\label{bynot}
\Sigma_j:= \sum _{\textbf{t} \rightarrow j } \Pi^{\bf P}(\textbf{t})  \quad \mbox{and} \quad  \Sigma^{(1)} := \sum_{j \in S} \Sigma_j
\end{align}
follows readily from the {\em Markov chain tree theorem} \cite{FV,Shu,KV,LR,LRbis}:
\begin{theorem}[Markov chain tree theorem for irreducible chains] \cite{FV,Shu,KV}
\label{MCTT}
Assume that $\textbf{P}$ is irreducible or equivalently, that $\Sigma_j > 0$ for every $j \in S$. Then
\begin{equation}
\label{41}
\sum_{i \in S} \Sigma_i p_{ij} = \Sigma_j \quad \mbox{for all}~j \in S.
\end{equation}
Consequently, the unique stationary distribution for the chain is $\pi_j = \Sigma_j / \Sigma^{(1)}$. 
\end{theorem}
\quad Section \ref{kem} recalls the short combinatorial proof of this result due to Ventcel and Freidlin \cite{FV}, where Theorem \ref{MCTT} appeared as an auxiliary lemma to study random dynamical systems. It was also formulated by Shubert \cite{Shu} and Solberg \cite{Solberg} in the language of graph theory, and by Kolher and Vollmerhaus \cite{KV} in the context of biological multi-state systems. 
The name {\em Markov chain tree theorem} was first coined by Leighton and Rivest \cite{LR,LRbis}, where they extended the result to general Markov chains which are not necessarily irreducible, see Theorem \ref{MC2}. 

\quad Later Anantharam and Tsoucas \cite{AT}, Aldous \cite{Aldous} and Broder \cite{Broder} provided probabilistic arguments by lifting the Markov chain to its spanning tree counterpart. A method to generate random spanning trees, the {\em Aldous-Broder algorithm}, was devised as a by-product of their proofs: see Lyons and Peres \cite[Section $4.4$]{LP}. See also Kelner and M\textpolhook{a}dry \cite{KMcs}, and M\textpolhook{a}dry, Straszak and Tarnawski \cite{MST} for development on fast algorithms to generate random spanning trees.
Recently, Biane \cite{Biane}, and Biane and Chapuy \cite{Bianec} studied the factorization of a polynomial associated to that spanning tree-valued Markov chain.  
Gursoy, Kirkland, Mason and Sergeev \cite{GKMS} extended the Markov chain tree theorem in the max algebra setting.

\quad As we discuss in Subsection \ref{ttr}, the Markov chain tree theorem is a probabilistic expression of {\em Kirchhoff's matrix tree theorem} \cite{Kir,Tutte, Chaiken}. See also Seneta \cite[Lemma 7.1]{Seneta} for a weaker form of this theorem and its application to compute stationary distributions of countable state Markov chains from finite truncations. Here is a version of Kirchhoff's matrix tree theorem, essentially due to Chaiken \cite{Chaiken} and Chen \cite{Chen}. We follow the presentation of Pokarowski \cite[Lemma $1.1$ and $1.2$]{Poka1}.
\begin{theorem}[Kirchhoff's matrix forest theorem for directed graphs]  \cite{Chaiken,Chen,Poka1}
\label{CCP}
 Let $R$ be a subset of the finite state space $S$ of a Markov chain $(X_n)_{n \in \mathbb{N}}$ with transition matrix $\textbf{P}$. Let $\textbf{L}:=\textbf{I}-\textbf{P}$ where $\textbf{I}$ is the identity matrix on $S$, and let $\textbf{L}(R)$ be the matrix indexed by $S \setminus R$ obtained by removing from ${\bf L}$ all the rows and columns indexed by $R$. Then
 \begin{align}
\label{mtdet}
\det \textbf{L}(R) = w(R):= \sum_{\roots(\textbf{f}) = R} \Pi^{\bf P}(\textbf{f}),
\end{align}
where the sum is over all forests $\textbf{f}$ labeled by $S$ whose set of roots is $R$.
Moreover,
\begin{align}
\label{mtdet2}
\mbox{if}~  \det \textbf{L}(R) > 0, \quad  \mbox{then}~ \textbf{L}(R)^{-1} =  \left( \frac{ w_{ij}(R \cup \{j\} ) }{ w(R) } \right)_{i,j \in S \setminus R}
\end{align}
where 
\begin{align}
\label{wijr}
w_{ij}(R \cup\{j\})  := \sum_{\roots(\textbf{f}) = R \cup \{j\}, i \leadsto j} \Pi^{\bf P}(\textbf{f})
\end{align}
is the $\textbf{P}$-weight of all forests $\textbf{f}$ with roots $R \cup \{j\}$ in which the tree component containing $i$ has root $j$. 
\end{theorem}
\quad In the above theorem, the set of roots $R$ may include single points with no incident edges. 
Theodore Zhu pointed that the r.h.s. of \eqref{mtdet} is the probability that the functional digraph induced by a ${\bf P}$-mapping (see Pitman \cite{Pit01}) is a forest with root set $R$. This implies that
$0 \leq \det {\bf L}(R) \leq 1$. 

\quad As observed by Pokarowski \cite[Theorem $1.2$]{Poka1}, the expressions of Theorem \ref{CCP} have the following probabilistic interpretations. Assume that $w(R)>0$. First of all, the matrix in \eqref{mtdet2} is the Green function of the Markov chain with transition matrix $\textbf{P}$ killed when it hits $R$. From this, we derive the {\em tree formula for the Green function of a Markov chain} or simply
\vskip 6pt
\noindent
\textbf{Green tree formula}
\begin{align}
\label{green}
\mathbb{E}_{i} \sum_{n=0}^{T_R-1} 1(X_n = j ) = \textbf{L}(R)^{-1}_{ij} =  \frac{ w_{ij}(R \cup \{j\} ) }{ w(R) }   \quad \mbox{for}~i,j \in S \setminus R,
\end{align}
where $T_R:=\inf\{n \geq 0; X_n \in R\}$ is the entry time to the set $R$. Summing over $j \in S \setminus R$ gives an expression for the mean first passage time
\begin{align}
\label{meanfpt}
\mathbb{E}_{i} T_R =  \frac{ \sum_{j \in S \setminus R} w_{ij}(R \cup \{j\} ) }{ w(R) }.
\end{align}
The $\mathbb{P}_i$ distribution of $X_{T_R}$ is given by a variant of \eqref{green}: the {\em tree formula for harmonic functions of a Markov chain} or simply
\vskip 6pt
\noindent
 \textbf{Harmonic tree formula}
\begin{align}
\label{harmonic}
\mathbb{P}_{i} (X_{T_R} = j ) = \frac{ w_{ij}(R) }{ w(R) }   \quad \mbox{for}~i \in S~\mbox{and}~j \in R,
\end{align}
where $w_{ij}(R) = w_{ij}(R \cup \{j\} ) $ is exactly as in \eqref{green} but now $j \in R$ so $R \cup \{j\} =R$. 

\quad It is well known that the formulas \eqref{green}-\eqref{harmonic} all follow from characterizations of the probabilistic quantities as the unique solutions of linear equations associated with the {\em Laplacian matrix} $\textbf{L}$. For example, let
$${\bf P}^{\hit}(R): = (\mathbb{P}_{i} (X_{T_R} = j ); i \in S \setminus R~\mbox{and}~j \in R).$$
The usual first step analysis implies that
\begin{equation} \label{ggh}
{\bf P}^{\hit}(R) ={\bf P}_{(S \setminus R) \times R} + {\bf P}_{(S \setminus R) \times (S \setminus R)} {\bf P}^{\hit}(R),
\end{equation}
where ${\bf P}_{R \times R'}$ is the restriction of ${\bf P}$ to $R \times R'$.
By letting ${\bf L}(R; R')$ be obtained by removing from ${\bf L}$ all the rows indexed by $R$ and all the columns indexed by $R'$, the equation \eqref{ggh} is written as 
$${\bf L}(R) {\bf P}^{\hit}(R) = - {\bf L}(R; S \setminus R).$$
Then the harmonic tree formula \eqref{harmonic} is easily deduced from the fundamental expressions \eqref{mtdet} and \eqref{mtdet2} in Theorem \ref{CCP}. 

\quad The purpose of this work is to provide combinatorial and probabilistic meanings of these tree formulas, without appeal to linear algebra. The formulas \eqref{green}-\eqref{harmonic} can be proved by purely combinatorial arguments. As an example, a combinatorial proof of the harmonic tree formula \eqref{harmonic} is given in Section \ref{wu}. In addition, the Green tree formula \eqref{green} and the harmonic tree formula \eqref{harmonic} are closely related to Wilson's algorithm, whose original proof \cite{Wilson} is combinatorial. In fact,
\begin{itemize}
    \item
    the Green tree formula \eqref{green} is derived from the harmonic tree formula \eqref{harmonic}, together with standard theory of Markov chains;
    \item 
    the harmonic tree formula \eqref{harmonic} is a consequence of the success of Wilson's algorithm;
    \item
    Wilson's algorithm follows from the Green tree formula \eqref{green} by a probabilistic argument due to Lawler \cite{Lawler}.
\end{itemize}
These arguments are presented in Sections \ref{wu} and \ref{wiltree}. 
We show in Subsection \ref{ttr} that even the formula \eqref{mtdet} can be deduced from the Markov chain tree theorem. 
\begin{figure}[h]
\includegraphics[width=0.7\textwidth]{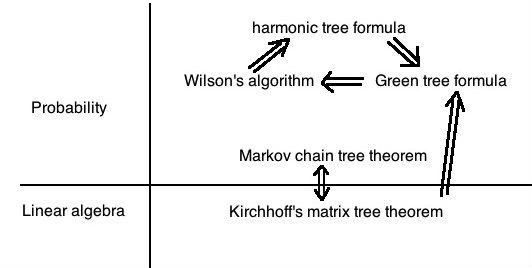}
\caption{Relations between various tree theorems/formulas.}
\end{figure}

\quad Theorem \ref{MCTT} provides a tree formula for the mean first passage time $m_{ij}$ for $i=j$. A companion result for $i \neq j$, which is a reformulation of Chebotarev \cite[Theorem $1$]{Che}, is stated as follows. The proof is deferred to Section \ref{kem}.
\begin{theorem}[Markov chain tree formula for mean first passage times]
\label{main}
Let ${\bf P}$ be a transition matrix for an irreducible chain. For each $i \ne j$,
\begin{align}
\label{mij}
m_{ij} = \Sigma_{ij}/\Sigma_j,
\end{align}
where
\begin{align}
\label{sigij}
\Sigma_{ij}:= \sum _{\textbf{t} \rightarrow j } \Pi^{\bf P}(\textbf{t})/p_{k(i,j,\textbf{t}) j},
\end{align}
with $k(i,j,\textbf{t})$ being the last state before $j$ in the path from $i$ to $j$ in $\textbf{t}$, and $\Sigma_j$ is defined by \eqref{bynot}.
\end{theorem}

\quad Observe that each term $\Pi^{\bf P}(\textbf{t})/p_{k(i,j,\textbf{t}) j}$ on the r.h.s. of \eqref{sigij} can be written
as
$$
\Pi^{\bf P}(\textbf{t})/p_{k(i,j,\textbf{t}) j} = \Pi^{\bf P}(\textbf{s},\textbf{u}),
$$
where $(\textbf{s},\textbf{u})$ is the forest of two trees obtained by deleting the edge $k(i,j,\textbf{t}) \rightarrow j$ from $\textbf{t}$. So the pair $(\textbf{s},\textbf{u})$ is a {\em two-component spanning forest}. It can easily be shown that the map $\textbf{t} \mapsto (\textbf{s},\textbf{u})$ is a bijection between trees $\textbf{t}$ with $\textbf{t} \rightarrow j$ and two tree forests $(\textbf{s},\textbf{u})$ such that $i \in \textbf{s}$ and $\textbf{u} \rightarrow j$. Thus the formula \eqref{sigij} for $\Sigma_{ij}$ can be rewritten as
\begin{align}
\label{cheba}
\Sigma_{ij} = \sum_{i \in \textbf{s}, \textbf{u} \rightarrow j } \Pi^{\bf P}(\textbf{s},\textbf{u}).
\end{align}

\quad Unaware of \cite{Che}, Hunter \cite{Hunter16} proposed an algorithm to compute mean first passage times in a Markov chain, and  derived the instances of \eqref{mij} for a Markov chain with two, three and four states. 
Note that for each $i \in S$, the number of terms in the sum $\Sigma_{ij}$ is the same as the number in the sum $\Sigma_j$, namely $|S|^{|S|-2}$. Also, each term $\Pi^{\bf P}(\textbf{t})/p_{k(i,j,\textbf{t})j}$ is larger than the corresponding term $\Pi^{\bf P}(\textbf{t})$ in $\Sigma_j$. To illustrate, for $|S|=2$ states $\{0,1\}$
\begin{align*}
m_{10} =1 / p_{10}.
\end{align*}
For $|S|=3$ states $\{0,1,2\}$
\begin{align*}
m_{10} =(p_{12}+p_{21}+p_{20}) / (p_{12}p_{20}+p_{21}p_{10}+p_{10}p_{20}).
\end{align*}

\quad It was first observed by Kemeny and Snell \cite[Corollary $4.3.6$]{KemenySnell} that the quantity 
\begin{equation}
\label{KSKS}
K:=\sum_{j \in S} m_{ij}/m_{jj}
\end{equation}
is a constant, not depending on $i$. This constant associated with an irreducible Markov chain is known as {\em Kemeny's constant}. Since its discovery, a number of interpretations have been provided. For example, Levene and Loizou \cite{LL} interpreted Kemeny's constant as the expected distance between two typical vertices in a weighted directed graph. Lovasz and Winkler \cite{LW} rediscovered this result in their {\em random target lemma}, which was further developed in Aldous and Fill \cite[Chapter $2$]{AldousFill}. 

\quad Kemeny's constant $K$ is closely related to the Laplacian matrix ${\bf L}$, and the {\em fundamental matrix} ${\bf Z}:=({\bf L}+{\bf \Pi})^{-1}$ where ${\bf \Pi}$ is the matrix each row of which is the stationary distribution $\pi$, by the following identities:
\begin{equation*}
K=\tr{\bf Z}; 
\end{equation*}
\begin{equation*}
K=\tr{[{\bf L}(i)^{-1}]}+\frac{{\bf L}^{\#}_{ii}}{\pi_i},
\end{equation*}
where ${\bf L}^{\#}$ is the {\em group inverse} of the Laplacian matrix ${\bf L}$. See also Doyle \cite{Doyle}, Hunter \cite{Hunter}, Gustafson and Hunter \cite{GH}, and Catral, Kirkland, Neumann and Sze \cite{CKNS} for linear algebra approaches to Kemeny's constant $K$.
The following result is a consequence of the formula \eqref{cheba} in the proof of Theorem \ref{main}.

\begin{corollary}[Combinatorial interpretation of {\em Kemeny's constant}]
\label{new}
For $r \in \mathbb{N}$,  let
$$
\Sigma^{(r)} :=  \sum_{\textbf{t}_1,\ldots, \textbf{t}_r} \Pi^{\bf P}(\textbf{t}_1, \ldots, \textbf{t}_r),
$$
where the sum is over all directed forests of $r$ trees $\textbf{t}_1,\ldots, \textbf{t}_r$ spanning $S$. Then
\begin{align}
\label{kemeny}
K = 1 + \Sigma^{(2)}/\Sigma^{(1)}.
\end{align}
\end{corollary}
\quad Hunter \cite{Hunter} indicated the instances of \eqref{kemeny} for a Markov chain with two and three states, but with a notation which conceals the generalization to $n$ states. So this combinatorial interpretation of $K$ may be new. We leave open the interpretation of $\Sigma^{(r)}$ for $r \geq 3$.

\vskip 6pt
\noindent
\textbf{Organization of the paper:}  
\begin{itemize}
\item
In Section \ref{wu}, we provide a combinatorial proof of the harmonic tree formula \eqref{harmonic}, from which we derive the Green tree formula \eqref{green}. We also prove Cayley's formula for enumerating spanning forests by means of the Green tree formula \eqref{green}.
\item
In Section \ref{kem}, we focus on the Markov chain tree theorems. We present a short proof of Theorem \ref{MCTT} and a generalization. We also provide two proofs for Theorem \ref{main}, one based on the formula \eqref{meanfpt} and the other relying on Theorem \ref{MCTT}.
\item
In Section \ref{Kmt}, we review Kirchhoff's matrix tree theorems. We show how to translate this graph theoretical result into the Markov chain setting. In particular, we show that the Markov chain tree theorem is derived from a version of Kirchhoff's matrix tree theorem.
\item
In Section \ref{wiltree}, we explore the relation between Wilson's algorithm and various tree formulas. We also present Kassel and Kenyon's generalized Wilson's algorithm, from which we derive some cycle-rooted tree formulas.
\item 
In Section \ref{bisbis}, some additional notes and further references are provided.
\end{itemize}
\section{Tree formulas and Cayley's formula}
\label{wu}
\quad In this section, we provide combinatorial and probabilistic proofs for the Green tree formula \eqref{green} and the harmonic tree formula \eqref{harmonic}. As an application, we give a proof of Cayley's well known formula \cite{Cayley} for enumerating spanning forests in a complete graph. 

\quad To begin with, we make a basic connection between results formulated for an irreducible Markov chain with state space $S$, and results formulated for killing of a possibly reducible Markov chain when it first hits an arbitrary subset $R$ of its state space. 
\begin{lemma}
\label{prelema}
Let ${\bf P}$ be a possibly reducible transition matrix indexed by a finite set $S$. For $R$ a non-empty subset of $S$, let $$w(R):= \sum_{\roots(\textbf{f}) = R} \Pi^{\bf P}(\textbf{f})$$ as in \eqref{mtdet}. The following conditions are equivalent:
\begin{enumerate}
\item \label{pos1} $w(R) >0$.
\item \label{pos2} There exists at least one forest $\textbf{f}$ of trees spanning $S$ with $\roots(\textbf{f}) = R$ such that the tree product $\Pi^{\bf P}(\textbf{f})>0$, which is to say, every edge $i \rightarrow j$ of $\textbf{f}$ has $p_{ij} >0$.
\item \label{pos3} For every $i \in S \setminus R$ there exists a path $\textbf{t}$ from $i$ to some $r \in  R$ such that $\Pi^{\bf P}(\textbf{t}) >0$
\item \label{pos4} $\textbf{L}(R):=(\textbf{I}-\textbf{P})_{(S \setminus R) \times (S \setminus R)}$ is invertible with inverse $\textbf{L}(R)^{-1} = \sum_{n=0}^\infty \textbf{P}_{(S \setminus R) \times (S \setminus R)}^n$, where $\textbf{P}_{(S \setminus R) \times (S \setminus R)}$ is the restriction of $\textbf{P}$ to $(S \setminus R) \times (S \setminus R)$.
\item \label{pos5} $\det\textbf{L}(R) \ne 0$.
\end{enumerate}
\end{lemma}
\begin{proof}
Note that $(1) \Leftrightarrow (2) \Rightarrow (3)$ is obvious. $(3) \Rightarrow (2)$ is obtained by recursively selecting a path until it either joins an existing path leading to some $r \in R$, or reaches a different $r' \in R$. The procedure terminates when all of $S \setminus R$ are exhausted. This point is reinforced by Wilson's algorithm, see Section \ref{wiltree}. As for $(3) \Leftrightarrow (4)$, this is textbook theory of absorbing Markov chains, see Kemeny and Snell \cite[Theorem $3.2.1$]{KemenySnell}, or Seneta \cite[Theorem $4.3$]{Seneta}. Finally, $(4) \Leftrightarrow (5)$ is elementary linear algebra. 
\end{proof}

\quad By Kirchhoff's matrix forest theorem, Theorem \ref{CCP}, we know that $w(R) = \det\textbf{L}(R)$, which is much more informative than the implication $(1) \Leftrightarrow (5)$ of Lemma \ref{prelema}.  But we are now trying to work around the matrix tree theorem, to increase our combinatorial and probabilistic understanding of its equivalence. 
Now we present a combinatorial proof of the harmonic tree formula.
\vskip 6pt
\noindent
{\bf Proof of the harmonic tree formula \eqref{harmonic}.}
Assume that $w(R) >0$. By Lemma \ref{prelema}, for every $i \in S$, there exists a path $\textbf{t}$ from $i$ to some $r \in  R$ such that $\Pi^{\bf P}(\textbf{t}) >0$. By finiteness of $S$ and geometric bounds, we have $\mathbb{P}_i(T_R < \infty) = 1$ for all $i \in S$. This condition implies that for each $r \in R$, the function
$$
h_r(i):= \mathbb{P}_i( X_{T_R} = r)
$$
is the unique function $h$ such that $h(i) = (\textbf{P}h)(i)$ for all $i \in S \setminus R$ with the boundary condition $h(i) = 1(i=r)$ for $i\in R$, see e.g. Lyons and Peres \cite[Section $2.1$]{LP}. Considering
$$
h(i):= \frac{w_{ir}(R)}{w(R)},
$$
it is obvious that this $h$ satisfies the boundary condition, so it only remains to check that it is $\textbf{P}$-harmonic  on $S \setminus R$. After canceling the common factor of $w(R)$ and putting all terms involving $w_{ir}$ on the left side, the
harmonic equation for $w_{ir}(R), i \in S \setminus R$ becomes
$$
\left( \sum_{j \ne i } p_{ij} \right) w_{ir}(R) = \sum_{k\ne i } p_{ik} w_{kr}(R).
$$
The equality of these two expressions is established by matching the terms appearing in the sums on the two sides. Specifically, for each fixed $i \in S \setminus R$ and $r \in R$ there is a matching
\begin{align}
\label{match}
p_{ij} \Pi(\textbf{f}) = p_{ik} \Pi(\textbf{f}'),
\end{align}
where on the l.h.s.: $j \ne i , i \stackrel{\textbf{f}}{\leadsto} r$ and on the r.h.s.: $k \ne i , k \stackrel{\textbf{f}'}{\leadsto} r$ with both $j$ and $k$ ranging over all states in $S$, but always $i \in S \setminus R$ and $r \in R$. 
If on the l.h.s. we have $j\stackrel{\textbf{f}}{\leadsto} r$, then set $k = j$ and $\textbf{f}'=\textbf{f}$. Then the r.h.s. conditions are met by $(k,\textbf{f}')$, and \eqref{match} holds trivially. So we are reduced to matching, for each fixed choice of $i \in S \setminus R$ and $r \ne r' \in R$, on the l.h.s.: $j \ne i , i \stackrel{\textbf{f}}{\leadsto} r, j \stackrel{\textbf{f}}{\leadsto} r' \neq r$ and on the r.h.s.: $k \ne i , k \stackrel{\textbf{f}'}{\leadsto} r, i \stackrel{\textbf{f}'}{\leadsto} r' \ne r$.
\begin{figure}[h]
\includegraphics[width=0.7 \textwidth]{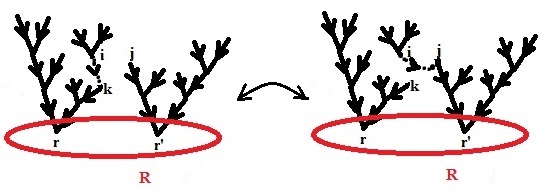}
\caption{Bijection between the l.h.s.: $i \leadsto r, j \leadsto r' $ and the r.h.s.: $k \leadsto r, i \leadsto r'$.}
\end{figure}

Given $(j,\textbf{f})$ on the l.h.s., let $i \rightarrow k$ be the edge out of $i$ in $\textbf{f}$. Create $\textbf{f}'$ by
deleting this edge and replacing it with $i \rightarrow j$. Then it is easily seen that $(k,\textbf{f}')$ is as required on the r.h.s., and it is clear that \eqref{match} holds. Inversely, given $(k,\textbf{f}')$ as on the r.h.s., let $i \rightarrow j$ be the edge out of $i$ in $\textbf{f}'$. Pop this edge and replace it with $i \rightarrow k$ to recover $(j,\textbf{f})$.  \hfill\(\square\)
\vskip 6pt
\noindent
\quad Next we make use of the harmonic tree formula to derive the Green tree formula. To this end, we need the following tree identity.
\begin{lemma}
\label{treealg}
For $j \in S \setminus R$,
$$w(R \cup \{j\}) = w(R) + \sum_{k \in S \setminus R}p_{jk}w_{kj}(R \cup \{j\}).$$
\end{lemma}
\begin{proof} 
Observe that
$$w(R \cup \{j\}) = \sum_{\roots({\bf f})=R, {\bf t} \rightarrow j} \Pi^{\bf P}({\bf f}) \Pi^{\bf P}({\bf t}),$$
where the sum is over all forests ${\bf f}$ whose set of roots is $R$, and all trees ${\bf t}$ are directed towards $j$. Now for each choice of $({\bf f}, {\bf t})$, we can split the product into two parts as
$$\Pi^{\bf P}({\bf f}) \Pi^{\bf P}({\bf t}) = \sum_{k \notin {\bf t}} \Pi^{\bf P}({\bf f}) \Pi^{\bf P}({\bf t}) p_{jk} + \sum_{k \in {\bf t}} \Pi^{\bf P}({\bf f}) \Pi^{\bf P}({\bf t}) p_{jk}.$$
The sum of the first part is evidently $w(R)$, with $\Pi^{\bf P}({\bf f}) \Pi^{\bf P}({\bf t}) p_{jk}$ comprising those terms in $w(R)$ indexed by forests ${\bf f}'$ where the subtrees of ${\bf f}'$ rooted at $j$ equals ${\bf t}$ and that subtree is attached to the remaining forest ${\bf f}$ at vertex $k \in \bf{f}$. While the sum of the second part is
\begin{align*}
\sum_{\roots({\bf f})=R, {\bf t} \rightarrow j} \sum_{k \in \{\bf t\}} \Pi^{\bf P}({\bf f}) \Pi^{\bf P}({\bf t}) p_{jk} &=
 \sum_{k \in S \setminus R} p_{jk} \left[\sum_{ \roots({\bf f}) = R, k \in {\bf t} \rightarrow j} \Pi^{\bf P}({\bf f}) \Pi^{\bf P}({\bf t})\right] \\
& = \sum_{k \in S \setminus R} p_{jk} w_{kj}(R \cup \{j\}),
\end{align*}
from which the desired result follows.
\end{proof}
\vskip 6pt
\noindent
{\bf Derivation of the Green tree formula \eqref{green} from the harmonic tree formula \eqref{harmonic}.}
It follows from standard theory of Markov chains that for all $i,j \in S \setminus R$,
\begin{align}
\mathbb{E}_i \sum_{n=0}^{T_R-1}1(X_n=j)  
                                                                         &=\mathbb{P}_i (X_{T_{R \cup \{j\}}}=j) \times \mathbb{E}_j \sum_{n=0}^{T_R-1}1(X_n=j) \notag\\
                                                                         &\label{ij}=\frac{w_{ij}(R \cup \{j\})}{w(R \cup \{j\})} \times \mathbb{E}_j \sum_{n=0}^{T_R-1}1(X_n=j),
\end{align}
where the last equality uses the harmonic tree formula \eqref{harmonic} for $j \in R \cup \{j\}$. In addition,
\begin{align*}
\mathbb{E}_j \sum_{n=0}^{T_R-1}1(X_n=j) &= 1+ \sum_{k \in S \setminus R} p_{jk} \times \mathbb{E}_k \sum_{n=0}^{T_R-1}1(X_n=j) \\
                                                                         &\stackrel{\eqref{ij}}{=} 1+  \frac{\sum_{k \in S \setminus R} p_{jk}w_{kj}(R \cup \{j\})}{w(R \cup \{j\})} \times \mathbb{E}_j \sum_{n=0}^{T_R-1}1(X_n=j),
\end{align*}
which together with Lemma \ref{treealg} implies that 
\begin{align}
\label{abcd}
\mathbb{E}_j \sum_{n=0}^{T_R-1}1(X_n=j) &= \frac{w(R \cup \{j\})}{w(R \cup \{j\})-\sum_{k \in S \setminus R}p_{jk}w_{kj}(R \cup \{j\})} \notag\\
                                                                              &=\frac{w(R \cup \{j\})}{w(R)}.
\end{align}
Injecting \eqref{abcd} into \eqref{ij}, we obtain the Green tree formula \eqref{green}.
\hfill\(\square\)
\vskip 6pt
\noindent
\quad We illustrate the Green tree formula \eqref{green}, by a derivation of Cayley's formula for the number of forests with a given set of roots. Cayley's formula is well known to be a direct consequence of Kirchhoff's matrix forest theorem, see e.g. Pitman \cite[Corollary 2]{Pitmanrecent}. The Green tree formula, while weaker than Kirchhoff's matrix forest theorem, still carries enough enumerative information about trees and forests to imply Cayley's formula.
\begin{corollary}[Cayley's formula] \cite{Cayley} 
\label{cay}
Let $1 \leq k \leq n$. Then
\begin{equation}
\label{cayley}
| \{\mbox{forests labeled by $[n]$ with root set $[k]$}\} |= k n^{n-k-1}.
\end{equation}
\end{corollary}

\begin{proof}
Consider the Markov chain generated by an i.i.d sequence of uniform random choices
from $S:=[n]$ and run the chain until the first time it hits a state $i \in R:=[k]$. The number of steps required is a geometric random variable $T_{n,k}$ with mean $n/k$. In addition, the expectation of the intervening number of steps with mean $(n-k)/k$ is equidistributed over the $n-k$ other states. Thus, the expected number of visits to each of these other states prior to $T_{n,k}$ is $1/k$. 

\quad Let $\textbf{P}_{n,k}$ be the $(n-k) \times (n-k)$ substochastic transition matrix with all entries equal to $1/n$. It follows immediately from the above observation that the corresponding Green matrix $(\textbf{I} - \textbf{P}_{n,k})^{-1}$ has entries $1 + 1/k$ along the diagonal, all other entries being identically equal to $1/k$. To illustrate, for $n=6$ and $k=2$,
$$
(\textbf{I} -\textbf{P}_{6,2})^{-1} = 
\begin{pmatrix} 
5/6 & -1/6 & -1/6& -1/6 \\ 
-1/6 & 5/6 & -1/6& -1/6 \\  
-1/6 & -1/6 & 5/6& -1/6 \\ 
-1/6 & -1/6 & -1/6& 5/6 \\ 
\end{pmatrix} ^{-1}
=
\begin{pmatrix} 
3/2 & 1/2 & 1/2& 1/2 \\ 
 1/2 & 3/2 & 1/2&  1/2 \\  
 1/2 &  1/2 & 3/2&  1/2 \\ 
 1/2 &  1/2 &  1/2& 3/2 \\ 
\end{pmatrix}.
$$
\quad Let $c(n,k)$ be the number of forests labeled by $[n]$ with root set $[k]$. Then for $k+1 \le i \ne j \le n$, the ratio of forest sums in the Green tree formula \eqref{green} is readily evaluated to give
\begin{align}
\label{cayrec}
\frac{1}{k} = \frac{ c(n,k+1) (k+1) ^{-1} n^{-(n-k-1)}} { c(n,k) n^{-(n-k)}},
\end{align}
where the denominator sums the $c(n,k)$ identical forest products $n^{-(n-k)}$ from the $k$-tree forests with root set $[k]$, while the numerator sums the $c(n,k+1)/(k+1)$ identical forest products $n^{-(n-k-1)}$ from all the $(k+1)$-tree forest products of trees with root set $\{j\} \cup [k]$ in which $i$ is contained in the tree with root $j$. Here the division by $(k+1)$ accounts for the fact that each tree has exactly one of $k+1$ distinct roots. The formula \eqref{cayrec} simplifies to 
\begin{align}
\label{cay3}
c(n,k+1) = \frac{(k+1) c(n,k) }{ k \, n } \quad \mbox{for}~1 \le k \le n-2.
\end{align}
Since the enumerations $c(n,n) = 1$ and $c(n,n-1) = n-1$ are obvious, Cayley's formula \eqref{cayley} for $c(n,k)$ follows immediately from \eqref{cay3}.  
\end{proof}

\quad We also refer to Lyons and Peres \cite[Corollary $4.5$]{LP} for a proof of Cayley's formula by Wilson's algorithm, and to Pitman \cite[Section $2$]{Pitforestvol} for that using the forest volume formula. Lyons and Peres' proof is similar in spirit to ours, and the relation between Wilson's algorithm and various tree formulas will be discussed in Section \ref{wiltree}.
\section{Markov chain tree theorems}
\label{kem}
\quad In this section, we deal with the Markov chain tree theorems. 
To begin with, we present a three-sentence proof of Theorem \ref{MCTT}, due to Ventcel and Freidlin \cite[Lemma $7.1$]{FV}.
They studied perturbed diffusion processes by Markov chain approximations, where Theorem \ref{MCTT} was used to estimate the first hitting time of the Markov chain to a set.
\vskip 6pt
\noindent
{\bf Proof of Theorem \ref{MCTT}.}
Multiply the r.h.s. of \eqref{41} by $\sum_k p_{jk} = 1$ and expand the tree sums $\Sigma_i$ and $\Sigma_j$ to include the extra transition factors:
$$\mbox{l.h.s.} = \sum_{i \in S}  \sum_{{\bf t} \rightarrow i} p_{ij}\Pi^{\bf P}({\bf t}) \quad \mbox{and} \quad \mbox{r.h.s.} =\sum_{k \in S} \sum_{{\bf t}\rightarrow j} \Pi^{\bf P}({\bf t}) p_{jk}.$$
Then both sides equal the sum of $\Pi^{\bf P}(\textbf{g})$ over all directed graphs $\textbf{g}$ which span $S$ and contain exactly one cycle including $j$, see Figure \ref{F1h}. Such graphs ${\bf g}$ are called {\em cycle-rooted spanning trees} (CRST), see Section \ref{wiltree} for definition.

\begin{figure}[ht]
\label{F1h}
\includegraphics[width=0.3 \textwidth]{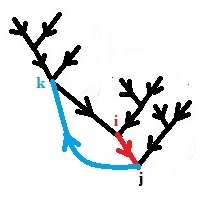}
\caption{A cycle-rooted spanning tree whose cycle includes $j$.}
\end{figure}

The l.h.s. sum is split up according to the state $i$ that precedes $j$, whereas the r.h.s. sum is split up according to the state $k$ that follows $j$. 
\hfill\(\square\)

\vskip 6pt
\noindent
\quad Now let us consider a Markov chain $(X_n)_{n \in \mathbb{N}}$ with transition matrix ${\bf P}$, which is not necessarily irreducible. 
Elementary considerations show that the state space $S$ is uniquely decomposed into a list of disjoint recurrent classes $C_1, \ldots , C_k$, and the transient states, see e.g. Feller \cite[Chapter XV.$6$]{Feller}.

\quad The general version of the Markov chain tree theorem is attributed to Leighton and Rivest \cite{LR,LRbis}. Here we provide a probabilistic argument, see also Anantharam and Tsoucas \cite{AT} for an alternative proof.
\begin{theorem}[Markov chain tree theorem] \cite{LR,LRbis} \label{MC2}
Assume that $\textbf{P}$ is a transition matrix with disjoint recurrent classes $C_1, \ldots , C_k$. Let $\mathcal{F}$ be a random forest picked from all forests $\textbf{f}$ consisting of $k$ trees with one tree rooted in each $C_i$, and $\mathbb{P}( \mathcal{F} = \textbf{f}) = \Pi^{\bf P}(\textbf{f}) / W$ where $W$ is the total weight of all such forests.
Then
\begin{equation}
\label{ddda}
\lim_{N \rightarrow \infty} \frac{1}{N} \sum_{n=1}^N p^n_{ij} = \mathbb{P}( i \stackrel{\mathcal{F}}{\leadsto} j),
\end{equation}
where the right hand side is the probability that the tree of $\mathcal{F}$ containing $i$ has root vertex $j$. 
\end{theorem}
\begin{proof} 
If $j$ is a transient state, then both sides of \eqref{ddda} equal to zero. Now let $C\in \{C_1,\ldots,C_k\}$ be the recurrent class containing $j$. 
According to Durrett \cite[Theorem 6.6.1]{Durrett},
$$
\lim_{N \rightarrow \infty} \frac{1}{N} \sum_{n=1}^N p^n_{ij}
= \mathbb{P}_i(T_C < \infty) \pi^C_j,
$$
where $T_C$ is the entry time to $C$ and $\pi_j^C$ is the stationary distribution in $C$.
The result boils down to two special cases: 
\begin{enumerate}
\item \label{case1}
the ergodic case when $\textbf{P}$ is irreducible and all trees in $\mathcal{F}$ have a single root, which is distributed according to the unique stationary distribution of ${\bf P}$ regardless of $i$;
\item \label{case2}
the completely absorbing case when there is a set $R$ of absorbing states with $w(R) >0$, and $\mathcal{F}$ is a forest whose set of roots is $R$.
\end{enumerate}
In case \eqref{case1}, the conclusion reduces to Theorem \ref{MCTT}, and in case \eqref{case2} to the harmonic tree formula \eqref{harmonic}. In the general case, every possible forest $\textbf{f}$ consists of 
\begin{itemize}
\item some selection of roots $R$, with $|R| = k$ and one root $r_i$ in each $C_i$,
\item for each $i$ a subtree $\textbf{t}_i$ spanning $C_i$ with root $r_i$,
\item a collection of subtrees $\textbf{u}_c$ rooted at $c \in  \cup_{i=1}^k C_i$.
\end{itemize}
Then an obvious factorization
$$
\Pi^{\bf P}(\textbf{f}) = \left( \prod_{i=1}^k \Pi^{\bf P}(\textbf{t}_i)  \right)  \prod_{c \in \cup_{i=1}^k C_i } \Pi^{\bf P}(\textbf{u}_c),
$$
shows that $\mathcal{F}$ decomposes into $(k+1)$ independent components, $k$ subtrees spanning the $C_i$, $1 \leq i \leq k$ and a forest with roots $\cup_{i=1}^k C_i$, so it is easy to deduce the conclusion from the two special cases.
\end{proof}

\quad The rest of this section concerns the Markov chain tree formula for mean first passage times, Theorem \ref{main}. First we provide a simple proof of Theorem \ref{main} by using the formula \eqref{meanfpt}, which is derived from the Green tree formula \eqref{green}.
\vskip 6pt
\noindent
{\bf Proof of Theorem \ref{main}.}
By setting $R=\{j\}$ in the formula \eqref{meanfpt}, we get:
$$m_{ij}=\frac{\sum_{k \neq j} w_{ik}(\{j,k\})}{w(\{j\}}.$$
By definition, $w(\{j\})=\Sigma_j$, and 
$$w_{ik}(\{j,k\})=\sum_{i \in {\bf s} \rightarrow k, {\bf u} \rightarrow j}\Pi^{\bf P}({\bf s},{\bf u}). $$
As a consequence, 
\begin{equation}
\label{bydef}
\sum_{k \neq j}w_{ik}(\{j,k\})=\sum_{i \in \textbf{s}, \textbf{u} \rightarrow j } \Pi^{\bf P}(\textbf{s},\textbf{u}).
\end{equation}
Combining \eqref{bydef} and \eqref{cheba} yields the desired result.
\hfill\(\square\)
\vskip 6pt
\noindent
\quad As observed by Pokarowski \cite{Poka1}, the Green tree formula \eqref{green} is a consequence of Kirchhoff's matrix forest theorem, Theorem \ref{CCP}. Now we give a combinatorial proof of Theorem \ref{main}, without appeal to the matrix tree theorem.
\vskip 6pt
\noindent
{\bf Alternative proof of Theorem \ref{main}.}
As seen in the beginning of this section, the formula \eqref{kirchhoff} can be proved without using Kirchhoff's matrix forest theorem. We now fix $i, j \in S$, and apply formula \eqref{kirchhoff} to the modified chain $\widetilde{{\bf P}}:=(\widetilde{p}_{ij};i,j \in S)$ defined by 
$$\tilde{p}_{ji}=1, \quad \tilde{p}_{jk}=0~\mbox{for}~k \neq i \quad \mbox{and} \quad \tilde{p}_{lk}=p_{lk}~\mbox{for}~l \neq j.$$

So $\widetilde{\bf P}$ has a recurrent class $\widetilde{S}$ containing $\{i,j\}$, and it is possible that $\widetilde{S} \neq S$. For $k \in \widetilde{S}$, let $\widetilde{\Sigma}_k$ be the tree sum, over trees $\widetilde{\bf t}  \rightarrow k$ spanning $\widetilde{S}$, of the tree product $$\Pi^{\widetilde{\bf P}}(\widetilde{\bf t}) = \prod_{ i' \rightarrow j' \in \widetilde{\bf t}} \tilde{p}_{i'j'}.$$

By construction, $\widetilde{m}_{jj} = 1 + m_{ij}$. Thus,
\begin{equation}
\label{mij0}
m_{ij} = \widetilde{m}_{jj}- 1  = \frac{\sum_{ k \ne j } \widetilde{\Sigma}_k}{\widetilde{\Sigma}_j}.
\end{equation}
\quad Observe that the map ${\bf t} \mapsto (\widetilde{\bf t}, {\bf f})$ is a bijection between trees {\bf t} with ${\bf t} \rightarrow j$ and tree forest pairs $(\widetilde{\bf t}, {\bf f})$ such that $\widetilde{\bf t} \rightarrow j$ spans $\widetilde{S}$ and $\roots({\bf f}) = \widetilde{S}$. This leads to a unique factorization $\Pi^{\bf P}({\bf t}) = \Pi^{\widetilde{\bf P}}(\widetilde{\bf t}) \Pi^{\bf P}({\bf f})$ for each ${\bf t} \rightarrow j$.
By summing over all ${\bf t} \rightarrow j$, we get
\begin{equation}
\label{mij1}
\Sigma_j = \widetilde{\Sigma}_j w(\widetilde{S}),
\end{equation}
where $w(\widetilde{S})$ is defined as in \eqref{mtdet}. 

\quad Further by cutting the edge $k(i,j,\textbf{t}) \rightarrow j$ from $\textbf{t}$, the map $\widetilde{\bf t} \mapsto ({\bf s}, {\bf u})$ is a bijection between trees $\widetilde{\bf t} \rightarrow j$ spanning $\widetilde{S}$ and two tree forests $(\textbf{s},\textbf{u})$ such that $i \in \textbf{s}$ and $\textbf{u} \rightarrow j$. Attaching ${\bf u}$ to ${\bf s}$ by $j \rightarrow i$, we define a tree $\widetilde{\bf t}'  \rightarrow k(i,j,\textbf{t})$ spanning $\widetilde{S}$. Since $\tilde{p}_{ji}=1$, we have
$$\Pi^{\widetilde{\bf P}}(\widetilde{\bf t})/p_{k(i,j,{\bf t})j} = \Pi^{\widetilde{\bf P}}(\widetilde{\bf t}').$$
Hence, $\Pi^{\bf P}({\bf t}) /p_{k(i,j,{\bf t})j}= \Pi^{\widetilde{\bf P}}(\widetilde{\bf t}') \Pi^{\bf P}({\bf f})$ for each ${\bf t} \rightarrow j$. Note that the map $\widetilde{\bf t} \mapsto \widetilde{\bf t}'$ is a bijection between trees $\widetilde{\bf t} \rightarrow j$ and those $\widetilde{\bf t}' \rightarrow k \neq j$. Again by summing over all ${\bf t} \rightarrow j$, we get
\begin{equation}
\label{mij2}
\Sigma_{ij} = \sum_{ k \ne j } \widetilde{\Sigma}_k w(\widetilde{S}).
\end{equation}
By injecting \eqref{mij1} and \eqref{mij2} into \eqref{mij0}, we obtain the formula \eqref{mij}.
\hfill\(\square\)
\vskip 6pt
\noindent
{\bf Chung's formula and tree identites.}  
Now by setting $R=\{k\}$ in the Green tree formula \eqref{green}, we get:
\begin{equation} \label{fromgreen}
\mathbb{E}_i \sum_{n=0}^{T_k-1} 1(X_n = j) = \frac{w_{ij}(\{k,j\})}{\Sigma_k} \quad \mbox{for}~i,j \neq k.
\end{equation}
According to Chung \cite[Theorem I.$11.3$]{Chung} and Pitman \cite[Example $4.11$]{Pitman77}, for a positive recurrent chain,
\begin{equation*} 
\mathbb{E}_i \sum_{n=0}^{T_k-1} 1(X_n = j) = \frac{m_{ik}+m_{kj} - m_{ij}1(i \neq j)}{m_{jj}} \quad \mbox{for}~i,j \neq k.
\end{equation*}
Further by Theorem \ref{main}, we have:
\begin{equation} \label{fromCP}
\mathbb{E}_i \sum_{n=0}^{T_k-1} 1(X_n = j) = \frac{\Sigma_{ik}\Sigma_j + \Sigma_{kj}\Sigma_k - \Sigma_{ij}\Sigma_k 1(i \neq j)}{\Sigma_k \Sigma^{(1)}} \quad \mbox{for}~i,j \neq k.
\end{equation}
By identifying \eqref{fromgreen} and \eqref{fromCP}, we obtain the following tree identities:
\begin{equation} 
w_{ii}(\{k,i\}) \Sigma^{(1)} = \Sigma_{ik}\Sigma_i + \Sigma_{ki} \Sigma_k \quad \mbox{for}~i \neq k,
\end{equation}
\begin{equation} 
w_{ij}(\{k,j\}) \Sigma^{(1)} + \Sigma_{ij} \Sigma_k = \Sigma_{ik}\Sigma_j + \Sigma_{kj} \Sigma_k \quad \mbox{for}~i \neq j~\mbox{and}~i,j \neq k.
\end{equation}
It seems that these identities are non-trivial, and there is no simple bijective proof. So we leave the interpretation open for readers.
\section{Kirchhoff's matrix tree theorems}
\label{Kmt}
\quad We begin with the discussion of {\em Kirchhoff-Tutte's matrix tree theorem for directed graphs}. Let $G:=(V,\overrightarrow{E})$ be a directed finite graph with no multiple edges nor self loops, where $\overrightarrow{E} \subset \{(i,j);i \neq j \in V\}$.
Equip each directed edge $(i,j) \in \overrightarrow{E}$ with a weight $c(i,j) \geq 0$. The {\em graph Laplacian} $\textbf{L}^{G,c}=(l^{G,c}_{ij};i,j \in V)$ is defined by: for $i,j \in V$,
$$l^{G,c}_{ij} :=     \left\{ \begin{array}{ccl}
         \sum_{k \neq i} c(i,k) & \mbox{if}\quad
          i=j \\ -c(i,j)  &\quad\quad \mbox{if}  \quad (i,j) \in \overrightarrow{E} \\
         0 &\quad \mbox{otherwise.}
                \end{array}\right.
$$
Observe that the graph Laplacian of a directed graph is not necessarily symmetric. If we take $c(i , j)=1$ for all $(i,j) \in \overrightarrow{E}$, then $\textbf{L}^{G,1}=\textbf{D}^{G}-\textbf{A}^G$, where $\textbf{D}^{G}$ is the {\em outer-degree matrix} of $G$, and $\textbf{A}^{G}$ is the {\em adjacency matrix} of $G$. 

\quad The following result, which we call {\em Kirchhoff-Chaiken-Chen's matrix forest theorem for directed graphs}, is due to Chaiken and Kleitman \cite{ChKl}, Chaiken \cite{Chaiken}, and Chen \cite{Chen}. 
\begin{theorem}[Kirchhoff-Chaiken-Chen's matrix forest theorem for directed graphs]\label{Kir3} \cite{ChKl, Chaiken, Chen}
\label{KCCMT} Let $R$ be a non-empty subset of $V$. Let $\textbf{L}^{G,c}(R)$ be the submatrix of $\textbf{L}^{G,c}$ obtained by removing all rows and columns indexed by $R$. Then
\begin{equation}
\label{nee}
\det\textbf{L}^{G,c}(R)=\sum_{\roots(\textbf{f})=R}\Pi^c(\textbf{f}),
\end{equation}
where $\Pi^c(\textbf{f}):= \prod_{(i,j) \in \textbf{f}}c(i , j)$, and the sum is over all forests whose set of roots is $R$.
\end{theorem}

\quad For further discussion on forest matrices, we refer to Chebotarev and Agaev \cite{Cheaga} and references therein. 

\quad The case $R = \{i\}$, which we call {\em Kirchhoff-Tutte's matrix tree theorem for directed graphs}, was first proved by Tutte \cite{Tutte} based on the inductive argument of Brooks, Smith, Stone and Tutte \cite{Brooks}. It was independently discovered by Bott and Mayberry \cite{Bott}. In particular,
$$\det \textbf{L}^{G,1}(i) =|\{\mbox{spanning trees rooted at}~i \in V\}|.$$
Orlin \cite{Orlin} made use of the  inclusion-exclusion principle to prove this result. Zeilberger \cite{Zeil} gave a combinatorial proof by using cancellation arguments. For historical notes, we refer to Moon \cite[Section $5.5$]{Moonbook}, Tutte \cite[Section VI.$4$]{Tuttebook} and Stanley \cite[Section $5.6$]{Stanleybook2}. There is a generalization of matrix tree theorems from graphs to simplicial complexes, initiated by Kalai \cite{Kalai} and developed by Duval, Klivans and Martin \cite{Duval}. Lyons \cite{Lyons3} extended the matrix tree theorem to CW-complexes. Minoux \cite{Minoux} studied the matrix tree theorem in the semiring setting. Masbaum and Vaintrob \cite{MV}, and Abdesselam \cite{AA} considered Pfaffian tree theorems. Recently, de Tili{\`e}re \cite{TBD} discovered a Pfaffian half-tree theorem. 

\quad It is well known that a weighted directed graph $(G,c)$ defines a Markov chain on the state $S:=V$. In the sequel, let $|S|=n$. The transition matrix $\textbf{P}:=(p_{ij}; i,j \in S)$ is given as
\begin{equation}
\label{GM}
p_{ij}:=-\frac{l^{G,c}_{ij}}{l^{G,c}_{ii}}=\frac{c(i , j)}{\sum_{k \neq i} c(i , k)} \quad \mbox{for}~i \neq j \in S \quad \mbox{and} \quad p_{ii}=0 \quad \mbox{for}~i \in S.
\end{equation}
Hence the transition matrix $\textbf{P}$ and the graph Laplacian $\textbf{L}^{G,c}$ are related by
\begin{equation}
\label{PL}
\textbf{L}^{G,c}=\textbf{D}^{G,c}(\textbf{I}-\textbf{P}),
\end{equation}
where $\textbf{D}^{G,c}$ is the diagonal matrix whose $(i,i)$-entry equals $l_{ii}^{G,c}$. Now we provide a proof of Theorem \ref{CCP} \eqref{mtdet} by using Theorem \ref{KCCMT}.
\vskip 6pt
\noindent
{\bf Derivation of Theorem \ref{CCP} \eqref{mtdet} from Theorem \ref{KCCMT}.}
The proof boils down to two subcases.

{\em Case 1.} For all $i \in S$, $p_{ii} = 0$. By the relation \eqref{PL}, the formula \eqref{mtdet} is an equivalent formulation of \eqref{nee} in the context of Markov chains.

{\em Case 2.} There exists $i \in S$ such that $p_{i i} > 0$.
Let $\mathcal{A}:=\{i \in S; p_{ii}=1\}$ be the set of absorbing states, and define a new chain $\widetilde{X}$ whose transition matrix $\widetilde{\bf P}=(\widetilde{p}_{ij}; i,j \in S)$ is given by $\widetilde{p}_{ii}=1$ if $i \in \mathcal{A}$ and
\begin{equation*}
\label{generaln}
\widetilde{p}_{ij}:=\frac{p_{ij}}{1-p_{ii}}  \quad \mbox{for}~i \neq j \quad \mbox{and} \quad \widetilde{p}_{ii}=0, \quad \mbox{if}~i \notin \mathcal{A},
\end{equation*}
Note that the restriction of the transition matrix $\widetilde{\bf P}$ to $(S \setminus \mathcal{A}) \times (S \setminus \mathcal{A})$ has all zeros on the diagonal.
If $\mathcal{A} \cap (S \setminus R) \neq \emptyset$, then both sides of \eqref{mtdet} equal to zero. Consider the case of $\mathcal{A} \subset R$.  Let $\widetilde{\bf L}:={\bf I}-\widetilde{\bf P}$, then \eqref{mtdet} holds for $(\widetilde{\bf L}, \widetilde{\bf P})$.
Multiplying both sides by $\prod_{j \notin R}(1-p_{jj})$ and noting that $\det {\bf L}(R)=\det \widetilde{\bf L}(R) \prod_{j \notin R}(1-p_{jj})$, we obtain \eqref{mtdet}.  
\hfill\(\square\)
\vskip 6pt
\noindent
\quad In the rest of this section, we present several applications of the matrix forest theorem.
\subsection{A probabilistic expression for $\Sigma^{(1)}$}
Recall the definition of $\Sigma^{(1)}$ from \eqref{bynot}. We present a result of Runge and Sachs \cite{RungeSachs} and Lyons \cite{Lyons1}, which expresses $\Sigma^{(1)}$ as an infinite series whose terms have a probabilistic meaning.

\quad By definition, 
\begin{equation}
\label{PLI}
\det(\textbf{I}-\textbf{P}-t\textbf{I})= \det(\textbf{L}-t {\bf I}) \quad \mbox{for all}~t \in \mathbb{R}.
\end{equation}
Let us look at the coefficient of $t$ on both sides of \eqref{PLI}. Let $\adj(\cdot)$ be the {\em adjugate matrix}. The coefficient of $t$ in $\det(\textbf{L}-t \textbf{I})$ is given by
\begin{equation*}
-\tr[\adj(\textbf{L})] =-\sum_{i=1}^{n} \det \textbf{L}(i) \stackrel{(*)}{=}-\sum_{i=1}^{n}  \sum_{\textbf{t}\rightarrow i}  \Pi^{\bf P}(\textbf{t})=-\Sigma^{(1)},
\end{equation*}
where $(*)$ is obtained by applying Theorem \ref{CCP} with $R = \{i\}$. 

\quad Assume that ${\bf P}$ is irreducible and aperiodic. Let $1=\lambda_0, \lambda_1, \cdots, \lambda_{n-1}$ be eigenvalues of the transition matrix $\textbf{P}$. It follows from {\em Perron-Frobenius theory} that $|\lambda_i|<1$ for $1\leq i \leq n-1$. See Meyer \cite[Chapter $8$]{Meyer00} for development. Then the coefficient of $t$ in $\det(\textbf{I}-\textbf{P}-t\textbf{I})$ is 
\begin{multline*}
-\prod_{i=1}^{n-1}(1-\lambda_i) =-\exp\left(\sum_{i=1}^{n-1} \log(1-\lambda_i) \right) =-\exp\left(-\sum_{i=1}^{n-1} \sum_{k \geq 1}\frac{\lambda_i^k}{k} \right) \\
                                                        =-\exp\left(- \sum_{k \geq 1}\frac{1}{k}(\tr\textbf{P}^k-1) \right)=-\exp\left[ -\sum_{k \geq 1}\frac{1}{k}\left( \sum_{i=1}^{n} p^{(k)}_{ii}-1\right) \right].
\end{multline*}
where $p^{(k)}_{ii}$ is the probability that the Markov chain starting at $i$ returns to $i$ after $k$ steps. Therefore,
\begin{equation}
\Sigma^{(1)} = \exp\left[ -\sum_{k \geq 1}\frac{1}{k}\left( \sum_{i=1}^{n} p^{(k)}_{ii}-1\right) \right].
\end{equation}
\subsection{Matrix tree theorems and Markov chain tree theorems}
\label{ttr}
We prove that the Markov chain tree theorem for irreducible chains, Theorem \ref{MCTT}, and Kirchhoff-Tutte's matrix tree theorem, Theorem \ref{CCP} \eqref{mtdet} for $R=\{i\}$ can be derived from each other. The argument is borrowed from Leighton and Rivest \cite{LR}, and Sahi \cite{Sahi}.
\vskip 6pt
\noindent
{\bf Derivation of Theorem \ref{MCTT} from Theorem \ref{CCP} \eqref{mtdet}.}
Observe that for an irreducible chain with Laplacian matrix $\textbf{L}:={\bf I}-\textbf{P}$, the stationary distribution $\boldsymbol{\pi}=(\pi_i)_{i \in S}$ is uniquely determined by
$$\boldsymbol{\pi} \textbf{L}(\textbf{1},i)={\bf E}_i \quad \mbox{for}~i \in S,$$
where $\textbf{L}(\textbf{1},i)$ is the matrix obtained from $\textbf{L}$ by replacing the $i$-th column by $\textbf{1}:=(1,\cdots,1)^T$, and ${\bf E}_i$ is the vector with a one in the $i$-th column and zeros elsewhere. By Cram\'{e}r's rule, 
$$\pi_i=\frac{\det \textbf{L}(i)}{\det \textbf{L}(\textbf{1},i)}=\frac{\det \textbf{L}(i)}{\sum_{j \in S} \det \textbf{L}(j)} \quad \mbox{for}~i \in S.$$
Let $(\Sigma_i)_{i \in S}$ and $\Sigma^{(1)}$ be defined as in \eqref{bynot}. According to the formula \eqref{mtdet} for $R=\{i\}$, 
$$ \pi_i=\frac{\Sigma_i}{\Sigma^{(1)}} \quad \mbox{for}~i \in S,$$
 which leads to Theorem \ref{MCTT}. 
 \hfill\(\square\)
 \vskip 6pt
\noindent
{\bf Derivation of Theorem \ref{CCP} \eqref{mtdet} for $R = \{i\}$ from Theorem \ref{MCTT}.}
By Theorem \ref{MCTT}, 
$$\Sigma_i = \frac{\det \textbf{L}(i)}{\sum_{j \in S} \det \textbf{L}(j)} \Sigma^{(1)} \quad \mbox{for}~i \in S.$$
Observe that $\det \textbf{L}(i)$, $\sum_{j \in S} \det \textbf{L}(j)$, $\Sigma_i$ and $\Sigma^{(1)}$ are all homogeneous polynomials of degree $n-1$ in variables $(p_{jk}; j \neq k \in S)$ with integer coefficients. Now we prove that
\begin{lemma} \label{prime}
For each $i \in S$, the polynomial $\det \textbf{L}(i)$ is irreducible.
\end{lemma}
\begin{proof}
Identify the set $S$ with $\{0,1,\cdots,n-1\}$. By symmetry, it suffices to consider $\det \textbf{L}(0)$. Note that for $1 \leq i,j \leq n-1$,
$${\bf L}(0)_{ij} = \left\{ \begin{array}{ccl}
-p_{ij} & \mbox{if}
& i \neq j, \\ \sum_{k \neq i} p_{ik} & \mbox{if} & i = j. 
\end{array}\right.$$
It is easy to check that $(p_{ij}; 1 \leq i \leq n-1, 0 \leq j \leq n-1, i \neq j) \mapsto ({\bf L}(0)_{ij}; 1 \leq i, j \leq n-1)$ is an invertible linear map. According to Bocher \cite[Section 61]{Bocher}, $\det \textbf{L}(0)$ is irreducible as a polynomial in the matrix entries.
\end{proof}
\quad By Lemma \ref{prime}, $\det \textbf{L}(i)$ and $\sum_{j \in S} \det \textbf{L}(j)$ do not have any common factor, since the terms of $\det \textbf{L}(i)$ are strictly included in the sum $\sum_{j \in S} \det \textbf{L}(j)$. It follows that
$$\Sigma_i =\lambda \det \textbf{L}(i) \quad \mbox{for some rational}~\lambda.$$ By considering the coefficient of $\prod_{k \neq i}p_{ki}$ on both sides, we obtain $\lambda=1$ as desired.
 \hfill\(\square\)
 \vskip 6pt
\noindent
\quad By decomposing the state space $S$ into recurrent classes and transient sets, a similar argument as above shows that the Markov chain tree theorem, Theorem \ref{MC2}, and Kirchhoff's matrix forest theorem, Theorem \ref{CCP} can be derived from each other.
\vskip 6pt
\noindent
\subsection{Matrix tree theorem for undirected graphs}
Consider the case where $c(i,j) = c(j,i)$ for $i \neq j \in V$, so that the graph Laplacian ${\bf L}^{G,c}$ is symmetric positive semi-definite. 
The following result, known as {\em Kirchhoff's matrix tree theorem for undirected graphs} \cite{Kir,Translation}, is easily derived from Theorem \ref{KCCMT}.
\begin{theorem}[Kirchhoff's matrix tree theorem for undirected graphs] \label{Kir1}\cite{Kir}
For $i,j \in V$, let $\textbf{L}^{G,c}(i;j)$ be the submatrix of $\textbf{L}^{G,c}$ obtained by removing the $i$th row and $j$th column. Then
\begin{equation}
\label{treeun}
\det \textbf{L}^{G,c}(i;j) = \sum_{t \in TREES} \Pi^c(t),
\end{equation}
where $\Pi^c(t):= \prod_{\{i',j'\} \in t}c(i',j')$, and the sum is over all unrooted spanning trees in $G$. In particular, 
$$\det\textbf{L}^{G,1}(i;j) = |\{\mbox{unrooted spanning trees in}~G\}|.$$
\end{theorem}

\quad We refer to Moon \cite[Section $5.3$]{Moonbook}, in which Theorem \ref{Kir1} was proved by using the {\em Cauchy-Binet formula}. The most classical application of Theorem \ref{Kir1} is to count unrooted spanning trees in the complete graph $K_n$, known as Cayley's formula \cite{Cayley}. See also Pak and Postnikov \cite[Section $3$]{PakP} for enumerating unrooted spanning trees by using the property of reciprocity for some tree-sum degree polynomial.

\quad  Theorem \ref{Kir1} states that every minor of the graph Laplacian $\textbf{L}^{G,c}$ is identical to the tree sum-product as in \eqref{treeun}. Kelmans and Chelnikov \cite{KC} expressed these minors in terms of the eigenvalues $0 = \lambda_0 \leq \lambda_1 \leq \cdots \leq \lambda_{n-1}$ of $\textbf{L}^{G,c}$:
\begin{equation}
\label{kc73}
\det \textbf{L}^{G,c}(i;j)=\frac{1}{n}\prod_{k=1}^{n-1} \lambda_k.
\end{equation}
\quad It was observed by Biggs \cite[Corollary $6.5$]{Biggsbook} that \eqref{kc73} can be derived from Temperley's identity \cite{Temp}: 
$
\det \textbf{L}^{G,c}(i;j)=\frac{1}{n^2} \det(\textbf{L}^{G,c}+J),
$
where $J$ is $n \times n$ matrix with all entries equal to $1$. Recently, Kozdron, Richards and Stroock \cite[Theorem $2.2$]{KRS} observed that the formula \eqref{kc73} is a direct consequence of  {\em Cr\'{a}mer's formula} and the {\em Jordan-Chevalley decomposition}, see also Stroock \cite[Section $3.2.2$]{Stroock}.

\quad Here we give a lesser known example of counting spanning trees in the complete prism. 
Recall that the Cartesian product $G \square H$ of graphs $G$ and $H$ is the graph such that $V(G \square H)=V(G) \times V(H)$, and $(u_G,u_H)$ is adjacent with $(v_G,v_H)$ if and only if $u_G=v_G$ and $u_H$ is adjacent to $v_H$ in $H$, or $u_H=v_H$ and $u_G$ is adjacent to $v_G$ in $G$.
\begin{example}[Boesch and Prodinger] \cite{BPch}
{\em We aim at counting spanning trees in the complete prism $K_n \square C_m$, that is the Cartesian product of the complete graph $K_n$ and the circulant graph $C_m$ whose adjacency matrix is a permutation matrix. The graph Laplacian of $K_n \square C_m$ is written as
\begin{equation*}
\textbf{L}^{K_n \square C_m,1}=\textbf{D}^{K_n \square C_m}-\textbf{A}^{K_n \square C_m}
                                                    =(n+1)I_{nm} -\textbf{A}^{K_n} \oplus \textbf{A}^{C_m},
\end{equation*}
where $\textbf{D}^{G}$ (resp. $\textbf{A}^{G}$) is the degree matrix (resp. the adjacency matrix) of G, and $\oplus$ is the Kronecker sum of two matrices: if $A$ is $m \times m$ matrix and $B$ is $n \times n$ matrix, then $$A \oplus B:=A \otimes I_n+I_m \otimes B, \quad \mbox{where}~\otimes~\mbox{is the usual tensor product of two matrices}.$$
Note that $\textbf{A}^{K_n}$ has eigenvalues $-1$ with multiplicity $n-1$, and $n-1$ with multiplicity $1$, and $\textbf{A}^{C_m}$ has eigenvalues $2 \cos(\frac{2k \pi}{m})$ for $0 \leq k \leq m-1$.
It is known that the eigenvalues of the Kronecker sum of two matrices are all possible sums of eigenvalues of the individual matrices, see Bellman \cite[Chapter 12, Section 11]{Bellman}. From Theorem \ref{Kir1} and \eqref{kc73}, we deduce that
\begin{align*}
&\quad |\{\mbox{unrooted spanning trees in}~K_n \square C_m\}| \\
&= \frac{1}{nm}\left(\prod_{k=0}^{m-1}\left[m+2-2 \cos\left(\frac{2k\pi}{m}\right)\right]\right)^{n-1} \prod_{k=1}^{m-1}\left[2-2\cos\left(\frac{2k\pi}{m}\right)\right]\\
&=mn^{n-2}\left[U_{m-1}\left(\sqrt{\frac{n+4}{4}}\right)\right]^{2n-2},
\end{align*}
where $U_{m-1}$ is the {\em Chebyshev polynomial of the second kind}; that is
$$U_{m-1}(x):=-\frac{1}{2 \sqrt{x^2-1}}\left[(x+\sqrt{x^2-1})^m-(x-\sqrt{x^2-1})^m\right].$$
See Boesch and Prodinger \cite{BPch}, Benjamin and Yerger \cite{BY}, and Zhang, Yong and Golin \cite{ZYG} for enumerating spanning trees in a wide class of graphs. }
\end{example}
\section{Wilson's algorithm and tree formulas}
\label{wiltree}
\quad In this section, we explore the connections between Wilson's algorithm, the Green tree formula \eqref{green} and the harmonic tree formula \eqref{harmonic}. Wilson's algorithm \cite{Wilson} was originally devised to generate a random tree whose probability distribution over trees $\textbf{t}$ rooted at a fixed $r \in S$ is proportional to the tree product $\Pi(\textbf{t})$. The constant of normalization is $w(\{r\})$ as in \eqref{mtdet} for the case $R= \{r\}$. See e.g. Lyons and Peres \cite[Section $4.1$]{LP}, and Grimmett \cite[Section $2.1$]{Grimmett} for further development.

\quad The algorithm has extensions in several directions. Marchal \cite{Marchal2,Marchal} provided a similar procedure to construct random Hamiltonian cycles. Gorodezky and Pak \cite{GPak} gave a version of Wilson's algorithm in the hypergraph setting. Kassel and Kenyon \cite{KK} generalized Wilson's algorithm for sampling cycle-rooted spanning forests, which will be discussed later.  

\quad For any finite path $(x_0, x_1, \ldots, x_l)$ in a directed graph, its {\em loop erasure} $(u_0,u_1, \ldots, u_m)$ is defined by erasing cycles in chronological order. More precisely, set $u_0 := x_0$. If $x_l = x_0$, we set $m = 0$ and terminate; otherwise, let $u_1$ be the first vertex after the last visit to $x_0$, i.e. $u_1 := x_{i+1}$, where $i := \max\{ j ; x_ j = x_0\}$. If $x_l = u_1$, then we set $m = 1$ and terminate; otherwise, let $u_2$ be the first vertex after the last visit to $u_1$, and so on.

\quad In the sequel, $(X_n)_{n \in \mathbb{N}}$ is a Markov chain with transition matrix $\textbf{P}:=\{p_{ij};i,j \in S\}$. Now we describe Wilson's original cycle-popping algorithm.
\vskip 6pt
\noindent
{\bf Wilson's algorithm for generating a random spanning tree rooted at $r$:}
Start a copy of $(X_n)_{n \in \mathbb{N}}$ at any arbitrary state $i$, run the chain until it hits $r$, and then perform a loop erasure operation to obtain a path from $i$ to $r$. This path will then be the unique path in the ultimately generated tree produced by following stages of the algorithm, in which another copy of $(X_n)_{n \in \mathbb{N}}$ is started at any arbitrary state not in this path, until it hits the path, and so on, growing an increasing family of trees which eventually span all of $S \setminus \{r\}$, when the algorithm terminates.

\vskip 6pt

\quad Next we consider Wilson's algorithm to produce random spanning forests.

\vskip 6pt

\noindent
{\bf Wilson's algorithm for generating a random spanning forest with roots $R$:} Start a copy of $(X_n)_{n \in \mathbb{N}}$ at any arbitrary state $i$, run the chain until it hits $j \in R$, and then perform a loop erasure operation to obtain a path from $i$ to $j \in R$. This path will then be the unique path in the ultimately generated forest produced by following stages of the algorithm, in which another copy of $(X_n)_{n \in \mathbb{N}}$ is started at any arbitrary state not in this path, until it hits either the path or a different $j' \in R$, and so on, growing an increasing family of forests which eventually span all of $S \setminus R$, when the algorithm terminates.

\begin{proposition}
\label{Wilsonforest}
Wilson's algorithm for generating a random forest spanning $S$ with roots $R$ terminates in finite time almost surely if and only if $w(R) >0$, in which case  Wilson's algorithm generates each forest $\textbf{f}$ with $\roots(\textbf{f}) = R$ with probability $\Pi^{\bf P}(\textbf{f})/ w(R)$.
\end{proposition}
\begin{proof}
This can be proved by simply adapting the cycle-popping argument of Wilson \cite{Wilson} to the present case. However, the result can also be derived from the more standard case of irreducible chains as follows. Let $\partial$ be an additional state not in $S$, and consider the  modified Markov chain with state space $\widetilde{S}:= S \cup \partial$ and transition matrix 
$$\widetilde{p}_{ij}:= \left\{ \begin{array}{ccl}
p_{ij}1(j\in S) & \mbox{if}
& i\in S \setminus R, \\ 1(j=\partial) & \mbox{if} & i \in R, \\
1(j \in S \setminus R) / |S \setminus R| & \mbox{if} & i=\partial.
\end{array}\right.$$
It is straightforward that $\widetilde{\textbf{P}}:=(\widetilde{p}_{ij}; i,j \in \widetilde{S})$ is irreducible if and only if $w(R) >0$. Also, if $\textbf{t}$ is a tree with root $\partial$ and $\textbf{f}$ is the restriction of $\textbf{t}$ to $S$, then $\Pi^{\widetilde{\bf P}}(\textbf{t}) = \Pi^{\bf P}(\textbf{f})$. Wilson's algorithm for generating a forest $\textbf{f}$ spanning $S$ with root set $R$ is now seen to be a variation of Wilson's algorithm to generate $\textbf{t}$ spanning $\widetilde{S}$ with root $\partial$. 
\end{proof}

\quad Avena and Gaudilli\`{e}re \cite{AG} were interested in random forests with random roots $\mathcal{R}$. They proved that under additional killing rates, the set of roots sampled by Wilson's algorithm is a determinantal process. Chang and Le Jan \cite{CLe} showed how Poisson loops arise in the construction of random spanning trees by Wilson's algorithm. Though closely related to ours, their situations seem to be more complicated.
\vskip 6pt
\noindent
{\bf Derivation of the harmonic tree formula \eqref{harmonic} from Wilson's algorithm.}
Observe that the harmonic tree formula \eqref{harmonic} is a consequence of the success of Wilson's algorithm for sampling $\textbf{f}$ with probability proportional to $\Pi^{\bf P}(\textbf{f})$. The first stage of Wilson's algorithm is to start a copy of the Markov chain $(X_n)_{n \in \mathbb{N}}$ at some $i \notin R$ and run it until time $T_R$. For each $r \in R$, the eventual forest $\textbf{f}$ generated by Wilson's algorithm has $r$ as the root of the tree containing $i$ if and only if this first stage results in $X_{T_R} = r$. Since $\textbf{f}$ ends up distributed with probability proportional to $\Pi^{\bf P}(\textbf{f})$, the formula \eqref{harmonic} is immediate. 
\hfill\(\square\)
\vskip 6pt
\noindent
\quad We have shown in Section \ref{wu} that the Green tree formula \eqref{green} can be deduced from the harmonic tree formula \eqref{harmonic}. Now we make use of the Green tree formula to prove Wilson's algorithm. These imply that Wilson's algorithm, the harmonic tree formula \eqref{harmonic} and the Green tree formula \eqref{green} can be derived from each other.
\vskip 6pt
\noindent
{\bf Derivation of Wilson's algorithm from the Green tree formula \eqref{green}.}
We borrow the argument which first appeared in Lawler \cite[Section $12.2$]{Lawler}, and was further developed by Marchal \cite{Marchal} and Kozdron, Richards and Stroock \cite{KRS}. See also Lawler and Limic \cite[Section $9.7$]{LawlerLimic} and Stroock \cite[Section $3.3$]{Stroock}.

\quad By the strong Markov property, the probability that $(i_1, \cdots, i_K)$ with $i_K \in R$ are successive states visited by loop-erased chain $(X_n; 0 \leq n \leq T_R)$ can be written as follows, using the notation $R_1:=R$ and $R_j:=R \cup \{i_1, \ldots, i_{j-1}\}$ for $2 \leq j \leq K$.
\begin{align}
\mathbb{P}^R(i_1,\cdots,i_K) &= \sum_{m_1, \ldots, m_{K-1}=0}^{\infty} \mathbb{P}_{i_1}(T_{i_1}<T_{R_1})^{m_1}p_{i_1i_2}\cdots \mathbb{P}_{i_{K-1}}(T_{i_{K-1}}<T_{R_{K-1}})^{m_{K-1}}p_{i_{K-1}i_K} \notag\\
                                                 &=\prod_{k=1}^{K-1} p_{i_k i_{k+1}} \frac{1}{1-\mathbb{P}_{i_k}(T_{i_k}<T_{R_k})} \notag\\
                                                 &=\prod_{k=1}^{K-1} p_{i_k i_{k+1}} \cdot \mathbb{E}_{i_k} \sum_{n=0}^{R_k}1(X_n=i_k) \notag\\
                                                 &\label{lawler}=\frac{w(R_{K-1})}{w(R)} \prod_{k=1}^{K-1} p_{i_k i_{k+1}},
\end{align}
where the last equality follows from the Green tree formula \eqref{green}. We initialize Wilson's algorithm by setting $V_0:=\{r\}$. Then define recursively $V_{l}$, the set of states visited up to $l^{th}$ iteration, and $\textbf{t}_{l}$ the tree branch added at $l^{th}$ iteration. From \eqref{lawler}, we deduce the probability for a spanning tree $\textbf{t}$ with root $r \in S$ generated by Wilson's algorithm 
$$\prod_{l \geq 1}\frac{w(V_l)}{w(V_{l-1})} \Pi^{\bf P}(\textbf{t}_{l})=\frac{w(S)}{w(\{r\})} \Pi^{\bf P}(\textbf{t})=\frac{\Pi^{\bf P}(\textbf{t})}{w(\{r\})}.$$
\hfill\(\square\)
\vskip 6pt
\noindent
\quad We conclude the section by presenting a generalized Wilson's algorithm, due to Kassel and Kenyon \cite{KK}, for sampling cycle-rooted spanning forests. To proceed further, we need some definitions. Let $G:=(V,E)$ be a directed finite graph, and $R$ be a subset of $V$. 
\begin{itemize}
\item
A {\em cycle-rooted spanning forest} (CRSF) in $G$ is a subgraph such that each connected component is a {\em cycle-rooted tree}, that is  containing a unique oriented cycle, and edges in the bushes (i.e. not in the cycles) being directed towards the cycle.
\item
An {\em essential cycle-rooted spanning forest} (ECRSF) of $(G,R)$ is a subgraph such that each connected component is either a tree directed towards a vertex in $R$, or a {\em cycle-rooted tree} containing no vertices in $R$. In particular, an ECRSF of $(G,\emptyset)$ is a CRSF.
\end{itemize}
\quad Now to each cycle $\gamma$ of the Markov chain $(X_n)_{n \in \mathbb{N}}$, we assign a  parameter of selection $\alpha(\gamma) \in [0,1]$. For an ECRSF $\textbf{f}_{ec}$ with vertex set $S$, call
\begin{equation*}
\Pi^{\textbf{P},{\bf \alpha}}(\textbf{f}_{ec}):=\prod_{i \rightarrow j \in \textbf{f}_{ec}}p_{ij} \prod_{\gamma \subset \textbf{f}_{ec}}\alpha(\gamma)
\end{equation*}
the $(\textbf{P},{\bf \alpha})$-weight of $\textbf{f}_{ec}$. 

\quad Kassel and Kenyon's generalized Wilson's algorithm generates a random ECRSF whose probability distribution over ECRSFs $\textbf{f}_{ec}$ with tree roots $R$ is proportional to $\Pi^{{\bf P}, \alpha}(\textbf{f}_{ec})$. The method is a refinement of the cycle-popping idea: we simply run Wilson's algorithm, and when a cycle $\gamma$ is created, flip a coin with bias $\alpha(\gamma) \in [0,1]$ to decide whether to keep or to pop it.
\vskip 6pt
\noindent
{\bf Kassel-Kenyon-Wilson's algorithm for generating an ECRSF with roots $R$:} Let $\Xi$ be a directed subgraph, initially set to be the tree roots $R$. Start a copy of $(X_n)_{n \in \mathbb{N}}$ at any arbitrary state $i \notin \Xi$, and run the chain until it first reaches a state $j$, which either belongs to $\Xi$ or creates a loop in the path. 
\begin{itemize}
\item
If $j \in R$, then replace $\Xi$ by the union of $\Xi$ and the path which is just traced. Start another copy of $(X_n)_{n \in \mathbb{N}}$ at $i \notin \Xi$ and repeat the procedure.
\item
If a loop $\gamma$ is created at the visit of the state $j$, then sample an independent $\{0,1\}$-Bernoulli random variable with success probability $\alpha(\gamma)$. 
\begin{itemize}
\item 
If the outcome is $1$, then replace $\Xi$ by the union of $\Xi$ and the path which is just traced. Start another copy of $(X_n)_{n \in \mathbb{N}}$ at $i \notin \Xi$ and repeat the procedure.
\item
If the outcome is $0$, then pop the loop and continue the chain until it hits $\Xi$ or a loop is created. In this case, repeat the above procedure.
\end{itemize}
\end{itemize}
\begin{theorem} \cite{KK} \label{KKW}
Kassel-Kenyon-Wilson's algorithm for generating a random essential cycle-rooted forest spanning $S$ with tree roots $R$ terminates in finite time almost surely if and only if at least one cycle has a positive parameter of selection, which is equivalent to
$$\sum_{\roots(\textbf{f}_{ec})=R}\Pi^{\textbf{P},{\bf \alpha}}(\textbf{f}_{ec})>0,$$
where the sum is over all ECRSFs labeled by $S$ with tree roots $R$. In this case, the algorithm generates each ECRSF $\textbf{f}_{ec}$ with tree roots $R$ with probability $\Pi^{\textbf{P},{\bf \alpha}}(\textbf{f}_{ec})/w^{ec}(R)$.
\end{theorem}
\quad When $\alpha = 0$, Theorem \ref{KKW} specializes to Wilson's algorithm for generation of spanning forests, Proposition \ref{Wilsonforest}. As Wilson's algorithm is related to various matrix tree theorems, Theorem \ref{KKW} has some affinity to {\em Forman-Kenyon's matrix CRSF theorem} \cite{Forman,Kenyon2011}. They proved that the determinant of the {\em line bundle Laplacian matrix} with {\em Dirichlet boundary} can be expressed as a certain (E)CRSF sum-product. See Kassel \cite{Kassel16} for derivation of Forman-Kenyon's matrix CRSF theorem from Theorem \ref{KKW}.

\quad Observe that if we set $\alpha(\gamma) = 1$ for all cycles $\gamma$, then $\Pi^{\textbf{P},1}(\textbf{f}_{ec})= \Pi^{\bf P}(\textbf{f}_{ec})$.
For $i \in S$ and $j \in R$, let 
$$w^{ec}_{ij}(R):=\sum_{\roots(\textbf{f}_{ec})=R, i \leadsto j}\Pi^{\textbf{P}}(\textbf{f}_{ec}),$$
be the $({\bf P},1)$-weight of all ECRSFs $\textbf{f}_{ec}$ with tree roots $R$, in which the tree containing $i$ has root $j$, and 
$$w^{ec}(R):=\sum_{\roots(\textbf{f}_{ec})=R}\Pi^{\textbf{P}}(\textbf{f}_{ec}).$$

Let $T_{\lp}$ be the first time at which a loop is created, and $T_R$ is the entry time of the set $R \subset S$. As a direct application of Theorem \ref{KKW}, we derive an analog of the harmonic tree formula:
\begin{equation}
\label{harmonic2}
\mathbb{P}_i(X_{T_R \wedge T_{\lp}}=j)=\frac{w^{ec}_{ij}(R)}{w^{ec}(R)} \quad \mbox{for}~ i \in S~\mbox{and}~j \in R,
\end{equation}
and by summing over $j \in R$,
\begin{equation}
\label{harmonic3}
\mathbb{P}_i(T_R < T_{\lp}) = \frac{\sum_{j \in R}w^{ec}_{ij}(R)}{w^{ec}(R)} \quad \mbox{for}~ i \in S.
\end{equation}
\section{Loose ends and further references}
\label{bisbis}
\subsection{Spanning trees and other models}
Spanning trees are closely related to various mathematical models. The connection between uniform spanning trees and loop-erased random walks has been discussed in Section \ref{wiltree}. Here we list two more examples.
\begin{itemize}
\item 
Temperley \cite{Temp1,Temp2} established a bijection between spanning trees of a square grid and perfect matchings/dimer coverings of a related square grid. This bijection was extended to general planar graphs by Burton and Pemantle \cite{BurtonPe}, and Kenyon, Propp and Wilson \cite{KPW}.
As an analog of Kirchhoff's matrix tree theorem for enumerating spanning trees, Kasteleyn--Temperley--Fisher's theory \cite{Kast,TempFi} expresses the number of dimer configurations in a graph in terms of the determinant of the {\em Kasteleyn matrix}. See Wu \cite{Wu} and Kenyon \cite[Section $3$]{Kenyonlec} for further development.
\item
Dhar \cite{Dhar}, and Majumdar and Dhar \cite{MDphi} constructed a bijective map between spanning trees of a graph and recurrent sandpiles on that graph. This map is not unique, and an alternative bijection was provided by Cori and Le Borgne \cite{CoriLe}. See J\'{a}rai \cite{Jarai} for development on sandpile models. Recently, Kassel and Wilson \cite{KasselWilson} have developed a new approach to computing sandpile densities of planar graphs by using a two-component forest formula of Liu and Chow \cite{LiuChow}.
\end{itemize}
See also Sokal \cite{Sokal} for how spanning trees arise as the limit of {\em $q$-Potts model} as $q \rightarrow 0$, and Bogner and Weinzierl \cite{BW10} for the use of spanning trees in the quantum field theory.
\subsection{Kemeny's constant and enumerate of spanning forests}
We start with a simple example of Corollary \ref{new}. Let $(X_k)_{k \in \mathbb{N}}$ be a Markov chain with state space $S: = [n]$, and transition matrix ${\bf P} = (p_{ij}; i,j \in [n])$ defined by 
$$
p_{ij} := \frac{1}{n} \quad \mbox{for all}~i,j \in [n].
$$
Let $K$ be Kemeny's constant defined by \eqref{KSKS} for the chain $(X_k)_{k \in \mathbb{N}}$.
By Corollary \ref{new},
\begin{equation}
\label{motex}
K = 1 + n \cdot  \frac{|\{\mbox{rooted two-component forests spanning}~S\}|}{|\{\mbox{rooted trees spanning}~S\}|}.
\end{equation}
According to Corollary \ref{cayley} (Cayley's formula),
$$|\{\mbox{rooted trees spanning}~S\}| = n^{n-1},$$
and
$$|\{\mbox{rooted two-component forests spanning}~S\}| = \binom{n}{2}\cdot 2n^{n-3},$$
which yields $K = n$. 

\quad Generally, we consider a weighted directed graph $(G,c)$. Define the ratio
\begin{equation}
\label{KDbis}
R(G,c) : =  \frac{\Sigma_c^{(2)}}{\Sigma_c^{(1)}},
\end{equation}
with
\begin{equation}
\label{KD}
    \Sigma_c^{(r)}: = \sum_{{\bf t_1},\ldots,{\bf t_r}} \Pi^c({\bf t_1},\ldots,{\bf t_r}) \quad \mbox{for}~r \in \mathbb{N},
\end{equation}
where the sum is over all forests of $r$ directed trees ${\bf t_1},\ldots,{\bf t_r}$ spanning $G$. 
So 
$$R(G,1) = \frac{n-1}{n} \sim 1 \quad \mbox{as}~n \rightarrow \infty.$$
In particular, the relation \eqref{motex} leads to $K = 1 + n R(G,1) = n$.

\quad Similarly, for a weighted undirected graph $(G,c)$, we define the ratio $R'(G,c)$ as in \eqref{KDbis}-\eqref{KD} but with the sum over all unrooted trees/forests spanning $G$. In this case,
$$R'(G,1) =  \frac{|\{\mbox{unrooted two-component forests spanning}~S\}|}{|\{\mbox{unrooted trees spanning}~S\}|}.$$
By an obvious bijection between unrooted spanning trees and rooted spanning trees with a particular root, we get
$$|\{\mbox{unrooted trees spanning}~S\}| = n^{n-2}.$$
Following from Moon \cite[Theorem 4.1]{Moonbook}, we have
$$|\{\mbox{unrooted two-component forests spanning}~S\}| = \frac{n^{n-4}(n-1)(n+6)}{2}.$$
Consequently,
$$R'(G,1) = \frac{(n-1)(n+6)}{2n^2} \sim \frac{1}{2} \quad \mbox{as}~n \rightarrow \infty.$$

Recently, Kassel and Wilson \cite{KasselWilson} considered the case where $G$ is a planar graph, and derived the asymptotic of $R'(G,c)$ as $n \rightarrow \infty$. See also Kenyon and Wilson \cite{KW15}, and Kassel, Kenyon and Wu \cite{KKW} for related results.
\subsection{Asymptotic enumeration of spanning trees} We have seen that Kirchhoff's matrix tree theorem enumerates explicitly spanning trees in finite graphs. It is interesting to understand the asymptotics of the number of spanning trees of a sequence of finite graphs that ``approach" an infinite graph.

\quad The following result was proved by Lyons \cite{Lyons1,Lyons2}. Let $(G_n; n \in \mathbb{N})$ be a sequence of finite connected graphs with bounded average degree, converging in the local weak sense of Benjamini--Schramm \cite{BenSh} to a probability measure $\rho$ on an infinite graph $G$. Then 
\begin{equation}
\lim_{n \rightarrow \infty} \frac{1}{|V(G_n)|} \log |\{\mbox{unrooted spanning trees in}~ G_n\}|={\bf h}(\rho),
\end{equation}
where ${\bf h}(\rho)$ is called the {\em tree entropy} of $\rho$ on $G$. If $\rho$ is unimodular, then
$
{\bf h}(\rho)=\log \det_{\rho} {\bf L}^G,
$
where $\det_{\rho} {\bf L}^G$ is the {\em Fuglede-Kadison determinant} of the graph Laplacian ${\bf L}^G$. The notion of Fuglede-Kadison determinant originates from von Neumann algebra, see e.g. de la Harpe \cite{Harpe} for a quick review.

\quad As explained by Lyons \cite{Lyons1}, the tree entropy ${\bf h}(\rho)$ appears as the entropy per vertex of a measure, which is the weak limit of the uniform spanning tree measures on $G_n$. We refer to Pemantle \cite{Pemantle}, Lyons \cite{Lyonsbird}, and Benjamini, Lyons, Peres and Schramm \cite{BLPS} for further discussion on the limiting uniform measures on spanning forests/trees.

\vskip 12pt
\noindent
\textbf{Acknowledgement:} We thank Yuval Peres for a helpful suggestion, which led to our alternative proof of Theorem \ref{main}. We are grateful to Jeffrey Hunter for informing us of the work \cite{Hunter16}, and to Abdelmalek Abdesselam for pointing out \cite{AA}. We thank Christopher Eur and Madeline Brandt for a discussion on irreducibility of polynomials. We also
thank two anonymous referees for their careful reading and valuable suggestions.
\bibliographystyle{plain}
\bibliography{Kemeny}
\end{document}